\documentclass[a4paper,12pt]{amsart}
\usepackage{graphicx}
\usepackage[applemac]{inputenc}
\usepackage{selinput}
\usepackage{color, hyperref}
\SelectInputMappings{
adieresis={ä},
ntilde={ñ}
}
\usepackage[spanish, english]{babel}
\usepackage{amssymb}
\usepackage{mathabx}
\usepackage{mathrsfs}
\usepackage{amsfonts,amscd,amsmath}
\usepackage{amssymb,latexsym}
\usepackage{enumitem} \numberwithin{equation}{section}

\def\irr#1{{\Irr}(#1)}
\def\irrp#1{{\Irr}_{p'}(#1)}

\def\galn#1#2{{\rm Gal}(#1/#2)}

\def\Oh#1#2{{\bf O}^{#1}(#2)}
\def\zent#1{{\bf Z}(#1)}
\def\syl#1#2{{\rm Syl}_{#1}(#2)}

\def\nor{\triangleleft}

\def\norm#1#2{{\bf N}_{#1}(#2)}
\def\cent#1#2{{\bf C}_{#1}(#2)}

\def\aut#1{{\rm Aut}(#1)}

\let\phi=\varphi
\def\sbs{\subseteq}

\newtheorem{lem}[subsection]{Lemma}

\newtheorem{teo}[subsection]{Theorem}

\newtheorem*{teo*}{Theorem}
\newtheorem*{conM}{The McKay Conjecture}

\theoremstyle{definition}
\newtheorem*{rem}{Remark}
\theoremstyle{definition}

\def\Q{\mathbb Q}
\def\Z{\mathbb Z}

\newcommand{\Irr}{\operatorname{Irr}}

\begin{document}

\author{Carolina Vallejo Rodr\'iguez}
\address{Departamento de Matem\'aticas,  
Universidad Aut\'onoma de Madrid, Campus de Cantoblanco, 28049 Madrid, Spain}
\email{carolina.vallejo@uam.es}


\thanks{The author is partially supported by the Spanish Ministerio
de Econom\'ia, Industria y Competividad proyectos MTM2016-76196-P and  MTM2017-82690-P, and the ICMAT Severo Ochoa
project SEV-2011-0087.}



\title[Character correspondences in solvable groups]{Character correspondences in solvable groups with a self-normalizing Sylow subgroup}

\begin{abstract} Let $G$ be a finite solvable group and let $P\in \syl p G$ for some prime $p$. Whenever $|G:\norm G P|$ is odd,  I. M. Isaacs described a correspondence between irreducible characters of degree not divisible by $p$ of $G$ and $\norm G P$. This correspondence is {\em natural} in the sense that an algorithm is provided to compute it, and the result of the application of the algorithm does not depend on choices made. In the case where $\norm G P= P$, G. Navarro showed that every irreducible character $\chi$ of degree not divisible by $p$ has a unique linear constituent $\chi^*$ when restricted to $P$, and that the map $\chi \mapsto \chi^*$ defines a bijection. Navarro's bijection is obviously natural in the sense described above. We show that these two correspondences are the same under the intersection of the hypotheses. \end{abstract}

\maketitle

\section*{Introduction}
In 1972, J. McKay foresaw that, for a finite simple group $G$, the number of odd-degree irreducible characters in $G$ is the same as the number of odd-degree irreducible characters in the normalizer of a Sylow 2-subgroup of $G$. This astonishing conjecture (which has been recently confirmed in \cite{MS16} for every finite group) was generalized by various authors to other families of groups and other primes until it took the form we know nowadays.
\begin{conM}Let $G$ be a finite group and let $P\in \syl p G$ for any prime $p$. Then
$$|\irrp G|=|\irrp {\norm G P}|,$$
where $\irrp G$ denotes the set of irreducible characters of $G$ of degree coprime to $p$.
\end{conM}

By saying a McKay bijection for $G$ and $p$ (where $p$ will be generally clear from the context) we mean a bijection $\irrp G \rightarrow \irrp{\norm G P}$.

\smallskip

In 1973, I. M. Isaacs created an algorithm to compute a McKay bijection whenever $G$ is solvable and $|G:\norm G P|$ is odd. Such algorithm is rather involved. In fact, if $\xi$ corresponds to $\chi$ under the algorithm, then $\xi$ is not always a constituent of the restriction $\chi_{\norm G P}$ of $\chi$ to $\norm G P$. (However, Isaacs showed that $\xi$ is a constituent of $\chi_{\norm G P}$ if $G$ has odd-order). Later on, precisely 30 years later, G. Navarro proved that a much simpler algorithm could be used to compute a McKay bijection whenever $G$ was a $p$-solvable group with a self-normalizing Sylow $p$-subgroup. He proved that for every $\chi \in \irrp G$, the restriction $\chi_P$ contains a unique linear character $\lambda_\chi \in \irr P$. The algorithm is clear in this situation, $\chi \mapsto \lambda_\chi$.

The question whether the application of these two algorithms gives the same result when the hypotheses meet has been raised by I. M. Isaacs and G. Navarro during the ``Group Representation Theory and Applications" program at the MSRI. 
In this short note, we positively answer this question in Theorem \ref{main}. The key is a character extension result, namely Theorem \ref{key}.

\smallskip

This is good news because natural bijections, as described in the abstract, are hoped to be unique. 

\smallskip

Throughout this note we follow the notation of \cite{Isa76} for ordinary characters. By a constituent of a character, we will mean an irreducible constituent. Indeed, if $\Delta$ is a character of some finite group $H$, then
$$\Delta=\sum_{\psi\in \irr H} a_\psi \psi,$$
for $a_\psi \in \Z_{\geq 0}$. We say that $\psi$ is a constituent of $\Delta$ if $a_\psi \neq 0$. 

If $N \nor G$ and $\chi \in \irr G$, then we say that $\theta$ lies under $\chi$ if $\theta$ is a constituent of $\chi_N$. We also say that $\chi$ lies over $\theta$ in this situation and write $\chi \in \irr{G|\theta}$.

\medskip

\noindent {\bf Acknowledgements.}
The author is indebted to the anonymous referee for valuable comments that improved the exposition of this paper. She would also like to thank Jon Gonz\'alez-S\'anchez and Gabriel Navarro for useful conversations on this subject.

\section{Isaacs' 1973 correspondece}
In 1973, only one year after the McKay conjecture was proposed, I. M. Isaacs described McKay bijections for solvable groups under a certain condition. In the beginning, he was interested in constructing a Glauberman-type character correspondence in the case of coprime actions on groups of odd-order. Not only did he succeed on that initial goal, but also he established the ideal way of solving the McKay conjecture. 

\begin{teo}[Isaacs, 1973]\label{Isaacs}
Let $G$ be solvable and let $P\in \syl p G$. Assume $|G:\norm G P |$ is odd. Then there exists a natural bijection between $\irrp G$ and $\irrp {\norm G P}$.
\end{teo}
The two main features of the above bijection are that it is equivariant with respect to the action of $\aut G _P$ on characters as well as equivariant with respect to the action of $\galn {\Q^{ab}}\Q$ on characters. (Although these are non-trivial facts, we do not intend to prove them here, since that would unnecessarily lengthen this exposition.)

The $\galn {\Q^{ab}}\Q$-equivariance implies that the multisets $\{ \{ \Q(\chi) \ | \chi \in \irrp G \}\}$ and $\{ \{ \Q(\xi) \ | \xi \in \irrp {\norm G P}  \}\}$ of field extensions of $\Q$ are the same. For a character $\psi$ of some group $H$, the field of values of $\psi$ over $\Q$ is defined as $\Q(\psi)=\Q(\psi(h) \ | \ h \in H)$. This property does not hold in general, not even in solvable groups, as shown by  ${\rm GL}(2,3)$ and $p=3$. However, G. Navarro proposed in 2004 that there should exist McKay bijections preserving fields of values over $\Q_p$, the field of $p$-adic numbers. This is known as the Galois-McKay conjecture, see \cite{Nav04}.

\smallskip

In general, the existence of McKay bijections equivariant with respect to the action of group automorphisms stabilizing the Sylow $p$-subgroup was conjectured by E. C. Dade (see Conjecture 9.13 of \cite{Nav18}).

\smallskip

There is a huge machinery underlying Theorem \ref{Isaacs} that culminated in the complete description of the character theory of fully ramified sections. This is done in Theorem 9.1 of \cite{Isa73}, and the description depends on a special character $\psi$ arising in such context. Note that $\psi$ does not generally contain the trivial character as a constituent. This fact can be already observed in the group ${\rm SL}(2,3)$.

\begin{rem}
As we have already mentioned in the Introduction, if $\chi\in \irrp G$ and $\xi\in \irrp{\norm G P}$ correspond under the bijection in Theorem \ref{Isaacs}, it is not true that $\xi$ is a constituent of $\chi_{\norm G P}$ (see the last paragraph of Section 9 of \cite{Isa73}). This situation is related to the character theory of fully ramified section briefly discussed above. We give next a specific example.

\smallskip

Let $K$ be extraspecial of order $3^3$ and exponent 3 
and let $H={\rm SL}(2,3)$. The group $H$ acts naturally on $K$ by fixing $L=\zent K$ and acting faithfully on $K/L$. Let $P \in \syl 2 H$, so that $\norm G P=L\times H$. If $1_L\neq \theta \in \irr L$, then $\theta$ is $G$-invariant and fully-ramified with respect to $K/L$. Let $\phi \in \irr K$ be over $\theta$. Since $\phi$ is $G$-invariant, then $\phi$ extends canonically to $\tilde \phi \in \irr{PK}$ by Corollary 6.28 of \cite{Isa76}. The extension $\tilde \phi$ is also $G$-invariant and thus extends to $G$ by Corollary 11.22 of \cite{Isa76}. In particular, $\phi$ extends to $G$. Then $\mathscr C=(G,K, \norm G P, L, \phi, \theta)$ is a constellation (with prime $\ell =3$) as in Definition 4.5 of \cite{Isa73} (the element $x \in \zent P$ of order 2 gives the desired automorphism of $G$). By Theorem 5.6 of \cite{Isa73}, the canonical character $\psi$ arising from $\mathscr C$ is a faithful, non-necessarily irreducible, character of $G/K\cong H$. By Theorem 3.5 \cite{Isa73}, $\psi$ has degree 3. In particular, $\psi$ must be the sum of a linear character and an irreducible character of  degree 2 of $G/K\cong H$.  Applying Theorem 5.7, and since $\cent {K/L} x=1$ for every $x \in P$, we get that $\psi(x)=-1$ whenever $x \in H$ has order 2, and $\psi(x)=1$ whenever $x \in H$ has order 4. By Corollary 6.3 of \cite{Isa73}, $\psi(x)=\pm\sqrt{- 3}$ whenever $x \in H$ has order 3. This latter equality forces the linear constituent of $\psi$ to be non-trivial. Two characters $\chi\in \irr{G|\phi}$ and $\xi \in \irr{\norm G P|\theta}$ of odd degree correspond under the bijection of Theorem \ref{Isaacs} if, and only if, $\chi_{\norm G P}=\psi_{\norm G P} \xi$ (as in Theorem 7.1 of \cite{Isa73}).

\smallskip

Another example of this situation was constructed in Section 3 of \cite{FN95}. There the authors were interested in characters correspondences arising under coprime actions of groups.
 \end{rem}

\section{Navarro's 2003 correspondence}
In 2003, G. Navarro studied McKay bijections under the key hypothesis that groups had a self-normalizing Sylow $p$-subgroup. 

Note that if a Sylow $p$-subgroup $P$ is self-normalizing in $G$, then 
$$\irrp {\norm G P}={\rm Lin}(P)=\{ \lambda \in \irr P \ | \lambda(1)=1\}\, .$$
\begin{teo}[Navarro, 2003]\label{Navarro} Let $G$ be $p$-solvable and let $P\in \syl p G$. Assume that $\norm G P =P$. If $\chi \in \irrp G$, then 
$$\chi_P=\chi^*+\Delta,$$
where $\chi^*\in {\rm Lin}(P)$ and $\Delta$ is either zero or every constituent of $\Delta$ has degree divisible by $p$. 
Furthermore the map defined by $\chi \mapsto \chi^*$ yields a bijection. 
\end{teo}

Navarro's bijection is natural. Moreover, it is straight-forward to verify that it is both $\aut G_P$ and $\galn{\Q^{ab}}\Q$-equivariant. 

\smallskip

It is worth mentioning that, already in 2003, Navarro suspected that the $p$-solvability condition above could be replaced by the condition that $p$ is odd. 
This prediction was confirmed in \cite{NTV14}, where 
the authors also provided a new method to prove Theorem \ref{Navarro}. 
The following extension result is the key of the new method. Notice that, for $p$-solvable groups, the conditions below on simple groups are superfluous.

\begin{teo}[Navarro, Tiep, Vallejo, 2014]\label{extension}
Let $G$ be a finite group, $p$ any prime, $P \in \syl pG$, and assume that $P=\norm GP$. Let $L \nor G$.
Let $\chi \in \irrp G$, and let $\theta \in \irr L$ be $P$-invariant under $\chi$. Assume that all non-abelian simple groups  of order divisible by $p$ involved 
in $L$ 
satisfy the {\sl inductive Alperin-McKay condition} for $p$.
Then $\theta$ extends to $G_\theta$. 
\end{teo}

 The inductive Alperin-McKay condition was defined in \cite{Spa13}. Fortunately, the inductive Alperin-McKay condition is not really necessary in this context. In \cite{NT16}, the authors prove the above extension result without conditions on simple groups for $p=2$. Indeed, by adapting the proof of Theorem 3.3 of \cite{NT16} for odd primes, the hypotheses on simple groups in Theorem \ref{extension} can be removed in general (see Theorem 3.6  of \cite{Val16}).

\section{Solvable groups with self-normalizing Sylow $p$-subgroups}
Suppose that $G$ is solvable and $\norm G P=P$ for some prime $p$, and either $p=2$ or $G$ has odd order (i.e. $|G:P|$ is odd). 
By Theorem \ref{Navarro}, if $\chi \in \irrp G$, then $\chi_P=\chi^*+\Delta$, where $\chi^*(1)=1$ and either $\Delta$ is zero or all its constituents 
have degree divisible by $p$. By Theorem \ref{Isaacs}, there is some $\xi \in \irr P$ linear, such that, $\chi$ and $\xi$ correspond. 
Are $\chi^*$ and $\xi$ the same character? This is the same as asking if $\xi$ is a constituent of $\chi_P$. Following the algorithm for constructing $\xi$ in Theorem 9.1 and Theorem 10.6 of \cite{Isa73}, one may suspect that problems can arise only if there are fully ramified situations. We show below (in Theorem \ref{key}) that fully ramified situations do not occur when computing $\xi$. 

\smallskip

We first list some elementary results that will be often used from now on. 

\begin{lem} Let $K\nor G$ and $H\leq G$ be such that $G=KH$ and $N=K\cap H$. Let $\phi \in \irr K$ be $G$-invariant and suppose that $\phi_N=\theta$. Then
restriction defines a bijection
$$\irr{G|\phi}\rightarrow \irr{H|\theta}.$$
\end{lem}
\begin{proof}  See Lemma 2.7 of \cite{Isa82}.
\end{proof}

\begin{lem}\label{Pinvariance} Let $N\nor G$ and let $\chi \in \irr G$. Assume that $\chi_N$ has some $P$-invariant constituent, where $P\in \syl p G$. Then every two of them are $\norm G P$-conjugate. In particular, this happens if $\chi$ has degree coprime to $p$.
\end{lem}
\begin{proof} Let $\theta \in \irr N$ be the $P$-invariant constituent of $\chi_N$  by hypothesis. Suppose that $\theta^x$ is $P$-invariant for some $x \in G$. Then $P, P^{x^{-1}}\sbs G_\theta$. Hence there exists some  $y\in G_\theta$ such that $P^{yx}=P$, note that $\theta^{yx}=\theta^x$ and $yx \in \norm G P$. 

If $\chi \in \irrp G$ and $\theta$ lies under $\chi$, then $|G:G_\theta|$ is a $p'$-number, hence $P^x\sbs G_\theta$ for some $x \in G$ so $\theta^{x^{-1}}$ is $P$-invariant. 
\end{proof}

We will also make use of the following consequences of the Glauberman correspondence. 

\begin{lem}\label{coprime} Let $P$ be a p-group that acts on a group $K$ stabilizing $N\nor K$. Assume that $K/N$ is a $p'$-group. 
Let  $\theta \in \irr N$ be $P$-invariant. Then $\theta^K$ has a $P$-invariant constituent, which is unique  if $\cent{K/N} P=1$.
\end{lem}
\begin{proof}Follows from Theorem 13.31 and Problem 13.10 of \cite{Isa76}.
\end{proof}

Having self-normalizing Sylow $p$-subgroups is related to the hypothesis on the centralizer in the above Lemma in the following way.

\begin{lem}\label{coprimeaction}
Suppose that $K\nor G$ is complemented by $H$. Then $H$ is self-normalizing in $G$ if, and only if, $\cent K H=1$. 
\end{lem}
\begin{proof} Note that $\norm G H \cap K=\norm K H=\cent K H$ since $[\norm K H,  H]\sbs H\cap K=1$.
\end{proof}

\smallskip

For the sake of completeness, we give an elementary proof of the $p$-solvable case of Theorem \ref{extension} following \cite{NT16}.
\begin{teo}\label{key}
Let $G$ be $p$-solvable and let  $N \nor G$. Suppose that $\norm G P =P$, where $P\in \syl p G$. If $\chi \in \irrp  G$ and $\theta \in \irr N$ lies under $\chi$, then $\theta$ extends to $G_\theta$.
\end{teo}
\begin{proof}
Since $\chi$ has $p'$-degree, some Sylow $p$-subgroup of $G$ is contained in $G_\theta$. By conjugating $\theta$ by some element of $G$, we may assume that $P\sbs G_\theta$. Hence we may assume that $\theta$ is $G$-invariant.  First note that  by the Frattini argument, whenever $M\nor G$, we have that
$$\norm {G/M}{PM/M}=\norm G P M /M.$$
Hence the hypothesis on the normalizer of the Sylow is inherited by quotients of $G$. We proceed by induction on $|N|$.

Suppose that $M\nor G$ and $1<M < N$. Let $\tau\in \irr M$ be $P$-invariant under $\chi$. In particular, $\tau$ lies under $\theta$. Moreover $G=NG_\tau$, since for every $g \in G$ we have that $\tau^g=\tau^n$ for some $n \in N$. First, if $G=G_\tau$ then by induction $\tau$ extends to $\eta \in \irr G$. By Gallagher's theorem $\chi=\alpha \eta$ for some $\alpha \in \irrp {G/M}$ and $\theta=\beta \eta_N$ for some $\beta \in \irr{N/M}$. Note that $\beta$ is $G$-invariant, again by Gallagher's theorem. Since $\beta$ lies under $\alpha$ of $p'$-degree and $G/M$ has self-normalizing Sylow $p$-subgroups, by induction $\beta$ extends to $\gamma \in \irr{G/M}$. Then $\gamma\eta \in \irr G$ satisfies $\gamma_N\eta_N=\beta\eta_N=\theta$, as wanted. Thus, we may assume that $G_\tau<G$, and so $N_\tau<N$. Let $\phi \in \irr{N_\tau|\tau}$ be the Clifford correspondent of $\theta$ and let $\psi\in \irr{G_\tau}$ be under $\chi$ and over $\phi$. In particular, $\psi$ is the Clifford correspondent of $\chi$ and $\tau$. Hence $\psi$ has $p'$-degree.  By induction, $\phi$ extends to $G_\tau=G_\phi$. Let $\eta \in \irr{G_\tau}$ extend $\phi$. Then $(\eta^G)_N=(\eta_{N_\tau})^N=\phi^N=\theta$ extends $\theta$, as wanted. 

By the previous paragraph, we may assume that $N$ is a minimal normal subgroup of $G$. If $N$ is a $p$-group, then $N \leq P$. Since $\chi$ has $p'$-degree, some constituent $\mu$ of $\chi_P$ has $p'$-degree. In particular, $\mu$ is linear and lies over $\theta$, hence $\mu_N=\theta$. If $Q \in \syl q G$ for some prime $q \neq p$, we have that $\theta$ extends to $QN$ by Corollary 6.20 of \cite{Isa76}. According to Corollary 11.31 of \cite{Isa76}, $\theta$ extends to $G$ as wanted. 
Otherwise, $N$ is a $p'$-group. Then the hypothesis $\norm G P =P$ implies that $\cent N P=1$ by Lemma \ref{coprimeaction}. By the Glauberman correspondence $\theta$ must be $1_N$, which trivially extends to $G$. 
\end{proof}

We are now ready to prove the main result of this note.

\begin{teo}\label{main}
Let $G$ be solvable and let $P\in \syl p G$. Assume that $\norm G P =P$ and either that $p=2$ or $G$ has odd order. Then the character correspondences given by Theorems \ref{Isaacs} and \ref{Navarro} are the same.
\end{teo}
\begin{proof} If $G$ is either a $p$-group or a $p'$-group, then there is nothing to prove, so $1<P<G$.

Given $\chi \in \irrp G$, the restriction $\chi_P$ has a unique linear constituent $\chi^*$ by Theorem \ref{Navarro}. Hence, in order to show that the two bijections are equal, it is enough to show that the character $\xi \in {\rm Lin }(P)$ corresponding to $\chi$ via Theorem \ref{Isaacs}  is a constituent of $\chi_P$. 

We compute $\xi$ by following the algorithm in Theorem 10.9 of \cite{Isa73}. Since $\norm G P=P$, we have that $\Oh{p'} G=G$ by the Frattini argument. Take $K=\Oh{p} G$ and $L=K'$, so that $L<K<G$, $G=KP$ and $K/L$ is an abelian $p'$-group. Write $H=PL$ so that $K\cap H=L$. Let $\theta \in \irr L$ be the unique $P$-invariant constituent of $\chi_L$ (by Lemma \ref{Pinvariance} together with the self-normalizing Sylow hypothesis)  and let $\eta \in \irr{G_\theta|\theta}$ be the Clifford correspondent of $\chi$. Note that $\chi_H=\eta_{H}+\Xi$, by MacKey's formula, where the constituents of $\Xi$ do not lie over $\theta$.

By Theorem \ref{key}, $\theta$ extends to $G_\theta$. Note that, by Gallagher's theorem, every character in $\irr{K_\theta|\theta}$ extends $\theta$. Note that $P$ acts on $K_\theta/L$ with $\cent{K_\theta/L} P=1$, by Lemma \ref{coprimeaction}. By Lemma \ref{coprime}, let $\phi$ be the unique $P$-invariant extension of $\theta$ to $K_\theta$, so that actually $\phi$ is $G_\theta$-invariant (using $G_\theta=K_\theta P$). Hence restriction defines a bijection $\irr{G_\theta|\phi} \rightarrow \irr{H | \theta}$.

Since $\eta$ has $p'$-degree, then $\eta_{K_\theta} \in \irr{K_\theta |\theta}$ by Corollary 11.29 of \cite{Isa76}. In particular, $\eta \in \irr{G_\theta|\phi}$, so $\eta_H\in \irr{H|\theta}$. Notice that we have proven that $\chi_{H}=\eta_{H}+\Xi,$ where $\eta_{H}\in \irrp{H|\theta}$ and the constituents of $\Xi$ do not lie over $\theta$.

Now $H$ is a group with a self-normalizing Sylow $p$-subgroup and we can repeat the above process with respect to $\eta_{H}\in \irrp H$. We can keep proceeding like this until we reach a linear character $\xi$ of $P$. At every step, we are choosing a constituent of the restriction of $\chi$ to some subgroup containing $P$, so in the end $\xi$ is a constituent of $\chi_P$, as wanted.   
\end{proof}

\end{document}